\documentclass[11pt]{article}
\usepackage[T1]{fontenc}
\usepackage{lmodern}
\usepackage{amsmath,amssymb,amsthm,mathtools}
\usepackage{microtype,needspace}
\usepackage{booktabs,array,longtable,enumitem}
\usepackage[textwidth=5.65in,textheight=8.4in,centering]{geometry}
\usepackage[hidelinks]{hyperref}
\allowdisplaybreaks[1]
\newtheorem{theo}{Theorem}[section]
\newtheorem{prop}[theo]{Proposition}
\newtheorem{coro}[theo]{Corollary}
\newtheorem{lem}[theo]{Lemma}
\theoremstyle{definition}

\newtheorem{remark}[theo]{Remark}
\numberwithin{equation}{section}
\renewcommand\d{\mathop{}\!\mathrm{d}}
\hypersetup{pdfauthor={Matthew H. Y. Xie and Zuo-Ru Zhang},pdftitle={Unimodality of Kazhdan--Lusztig polynomials of sparse paving matroids}}

\begin{document}
\title{Unimodality of Kazhdan--Lusztig polynomials\\
of sparse paving matroids}
\author{%
  \begin{minipage}{\dimexpr\textwidth-2\tabcolsep\relax}
    \centering
    \normalsize Matthew H. Y. Xie$^1$ and Zuo-Ru Zhang$^{2,*}$ \\[6pt]
    \small
    $^1$School of Mathematical Sciences,\\
    Tianjin University of Technology, Tianjin 300384, P. R. China\\[6pt]
    $^2$School of Mathematical Sciences,\\
    Hebei Normal University, Shijiazhuang 050024, P. R. China\\[6pt]
    Email: $^1$\texttt{xie@email.tjut.edu.cn},
           $^2$\texttt{zrzhang@hebtu.edu.cn}\\[3pt]
    $^*$Corresponding author.
  \end{minipage}%
}
\date{}
\maketitle
\begin{abstract}
We prove that the Kazhdan--Lusztig polynomial of a loopless sparse paving
matroid is unimodal. When the rank is at least five and the corank is
positive, we show that the coefficient sequence is strictly unimodal
except for three explicitly determined parameter triples, each giving two
equal maximal coefficients. We also determine the modes and classify the
degenerate loopless sparse paving matroids of positive rank
and corank.
\end{abstract}

\section{Introduction}
Elias, Proudfoot and Wakefield~\cite{EP2016} introduced the Kazhdan--Lusztig
polynomial of a matroid and conjectured that its coefficients are nonnegative.
Braden, Huh, Matherne, Proudfoot and Wang~\cite{BH2020} proved this conjecture using the intersection
cohomology of matroids. The stronger conjectures of log-concavity and
real-rootedness do not hold in general: Cheng and Liu~\cite{ChengLiu2026}
constructed, over every finite field, representable matroids whose
Kazhdan--Lusztig polynomials are not even unimodal. This leaves the question
of which natural classes of matroids still satisfy unimodality. We prove
that it holds for all loopless sparse paving matroids.

A matroid of rank $d$ is \emph{paving} if every circuit has at least $d$
elements. It is \emph{sparse paving} if both it and its dual are paving.
The size of this class gives our result an asymptotic consequence.
Let $s_n$ be the number of sparse paving matroids on a fixed $n$-element set,
and let $N_n$ be the number of all matroids on that set. Mayhew, Newman, Welsh and
Whittle~\cite{MN2011} conjectured that $s_n/N_n\to1$. Pendavingh and
van der Pol~\cite{PV2015} proved the weaker statement
\[
\lim_{n\to\infty}\frac{\log s_n}{\log N_n}=1.
\]
Together with our unimodality theorem, this implies that logarithmically
almost all matroids have unimodal Kazhdan--Lusztig polynomials; see
Corollary~\ref{cor:asymptotic}.

Sparse paving matroids also admit an explicit coefficient formula, which
relates their Kazhdan--Lusztig polynomials to those of uniform matroids.
Write $U_{m,d}$ for the uniform matroid of rank $d$ on $m+d$ elements.
Its Kazhdan--Lusztig polynomial is real-rooted for $m=1$
(see Gedeon, Proudfoot and Young~\cite{Gedeon2017}) and for
$2\leq m\leq15$ (Gao, Lu, Xie, Yang and Zhang~\cite{GL2021}).
Xie and Zhang~\cite{XZ2023}
proved log-concavity for all uniform matroids. Gao, Li, Xie, Yang and Zhang~\cite{GLXYZ2026}
established induced log-concavity for the equivariant Kazhdan--Lusztig
polynomials of uniform and $q$-uniform matroids. Each of these results implies
unimodality of the ordinary Kazhdan--Lusztig polynomials in the corresponding range.

Several families of graphic matroids also have real-rooted, and hence
unimodal, Kazhdan--Lusztig polynomials. Lu, Xie and Yang~\cite{LXY2022}
proved real-rootedness for the graphic matroids of fan graphs, wheel graphs
and squares of paths, and Zhang~\cite{Zhang2026} proved it for those of
thagomizer graphs $K_{1,1,n}$ and complete bipartite graphs $K_{2,n}$.

For a sparse paving matroid, the formula of
Lee, Nasr and Radcliffe~\cite{LN2021} expresses its Kazhdan--Lusztig
polynomial as the uniform polynomial minus a correction proportional to
the number of circuit-hyperplanes. Subtraction does not preserve unimodality, so the
uniform case alone does not imply the result proved here.

For a loopless matroid $M$, let $L(M)$ be its lattice of flats,
$\chi_M(t)$ its characteristic polynomial, and $\operatorname{rk}M$ its rank.
The Kazhdan--Lusztig polynomial $P_M(t)\in\mathbb Z[t]$ is uniquely determined
by the following conditions:
\begin{enumerate}[label=(\roman*),itemsep=2pt]
\item $P_M(t)=1$ if $\operatorname{rk}M=0$;
\item $\deg P_M(t)<\operatorname{rk}M/2$ if $\operatorname{rk}M>0$;
\item
$\displaystyle t^{\operatorname{rk}M}P_M(t^{-1})
=\sum_{F\in L(M)}\chi_{M_F}(t)P_{M^F}(t)$.
\end{enumerate}
Here $M_F$ denotes the restriction to $F$, and $M^F$ the contraction by $F$.
We consider Kazhdan--Lusztig polynomials only for loopless matroids.

A finite sequence $(c_0,\ldots,c_s)$ is \emph{unimodal} if, for some
$r\in\{0,\ldots,s\}$,
\[
c_0\leq\cdots\leq c_r\geq c_{r+1}\geq\cdots\geq c_s.
\]
It is \emph{strictly unimodal} if
\[
c_0<\cdots<c_r>c_{r+1}>\cdots>c_s.
\]
Either chain of inequalities may be empty.
An index $r$ is a \emph{mode} if $c_r=\max_{0\leq j\leq s}c_j$.
We apply these terms to a polynomial via its coefficient sequence ending
at its actual degree. Two adjacent equal maximal coefficients form a
\emph{plateau}.

In the statements below, $M$ is a loopless sparse paving matroid of rank $d$
on $m+d$ elements, and $\lambda$ is its number of circuit-hyperplanes.
\begin{theo}\label{thm:strict-unimodality}
The polynomial $P_M(t)$ is unimodal. If $d\geq5$ and $m\geq1$, it is
strictly unimodal unless
\[
(m,d,\lambda)=(4,5,3),\quad(5,5,25),\quad(6,5,66).
\]
In each of these three cases, its two nonconstant coefficients are equal.
\end{theo}

We also determine the exact locations of the modes of $P_M(t)$.
For $m\geq2$ and $d\geq5$, every mode lies in
$\{\lfloor d/3\rfloor,\lfloor d/3\rfloor+1\}$, and the first of these
indices is the unique mode when $d\not\equiv2\pmod3$.

\begin{theo}\label{thm:mode}
Suppose $m\geq2$ and $d\geq5$. Define
\[
\kappa(m,d)=
\begin{cases}
\lfloor d/3\rfloor+1,&d\leq3m-7\text{ and }d\equiv2\pmod3,\\[2pt]
\lfloor d/3\rfloor,&\text{otherwise}.
\end{cases}
\]
Then $\kappa(m,d)$ is a mode of $P_M(t)$ unless $(m,d,\lambda)$ satisfies
one of the conditions
\[
\begin{array}{c@{\qquad}c}
(m,d)&\lambda\\ \hline
(4,5)&\lambda\geq4,\\
(5,5)&\lambda\geq26,\\
(5,8)&\lambda\geq83,\\
(6,11)&\lambda\geq940.
\end{array}
\]
In each such case, the unique mode is $\lfloor d/3\rfloor$.
\end{theo}
The realizability of these exceptional parameters is discussed in
Section~\ref{subsec:realizability}. In particular, the thresholds in the
theorem are not assertions that every larger value of $\lambda$ is possible.

For a matroid of rank $d\geq1$, the maximum possible degree of $P_M(t)$
is $\lfloor(d-1)/2\rfloor$. As a by-product, we also classify the cases of
positive corank in which this degree is not attained; see
Theorem~\ref{thm:nondegeneracy}.

The sparse paving formula makes each difference between consecutive
coefficients affine in $\lambda$. Bounds on the number of circuit-hyperplanes
therefore reduce the sign analysis to comparisons between the uniform and
correction differences. After normalization, an affine recurrence in the
corank propagates inequalities from coranks two and three. Separate estimates
handle the boundary cases, including the comparison between the two possible
modes when $d\equiv2\pmod3$.

The rest of the paper is organized as follows.
Section~\ref{sec:recurrences} contains the coefficient formulas and
recurrences. The inequalities in Section~\ref{sec:unimodality} prove
unimodality, and Section~\ref{sec:mode} completes the determination of the
modes and plateaux. We treat non-degeneracy in
Section~\ref{sec:nondegeneracy}. The longer algebraic comparisons and the
explicit recurrence coefficients are collected in the appendices.

\section{Coefficient formulas and recurrences}\label{sec:recurrences}
\subsection{The sparse paving correction}
Let $\mathrm{SP}(d,m+d)$ denote the set of sparse paving matroids of rank
$d$ on $m+d$ elements. Throughout this section, $M\in\mathrm{SP}(d,m+d)$
has $\lambda$ circuit-hyperplanes. For $m\geq1$ and $d\geq2$, the
coefficient formula of Lee, Nasr and Radcliffe~\cite[Theorem~12 and Lemma~4]{LN2021} gives
\begin{equation}\label{eq:sp-kl-formula}
P_M(t)=P_{U_{m,d}}(t)
-\lambda\bigl(P_{U_{1,d}}(t)-P_{U_{1,d-1}}(t)\bigr).
\end{equation}
Write $u(m,d,i)=[t^i]P_{U_{m,d}}(t)$ and $a(m,d,i)=[t^i]P_M(t)$.
We set these coefficients to zero outside the degree range. Although the
notation for $a$ suppresses $\lambda$, this parameter is fixed whenever
coefficients of $M$ are compared. Equation~\eqref{eq:sp-kl-formula} becomes
\begin{equation}\label{eq:a-coeff}
a(m,d,i)=u(m,d,i)-\lambda\omega(d,i),
\qquad \omega(d,i)=u(1,d,i)-u(1,d-1,i).
\end{equation}
Lee, Nasr and Radcliffe~\cite[Corollary~16 and Theorem~17]{LN2021} give the bounds
\begin{equation}\label{eq:lambda-range}
0\leq\lambda\leq\frac1{\max\{m,d\}+1}\binom{m+d}{m}
\leq\lambda_{\max}(m,d),
\end{equation}
where the symmetric bound is denoted by
\begin{equation}\label{eq:lambda-max}
\lambda_{\max}(m,d)=\frac2{m+d+2}\binom{m+d}{m}.
\end{equation}
We will also occasionally use the sharper first bound.

\subsection{Coefficients of Kazhdan--Lusztig polynomials of uniform matroids}
For $m\geq1$ and $1\leq i\leq\lfloor(d-1)/2\rfloor$,
Gao, Lu, Xie, Yang and Zhang~\cite{GL2021} give
\begin{equation}\label{eq:u-def-sum}
u(m,d,i)=\sum_{h=0}^{m-1}
\frac1{d-i}\binom{h+i-1}{h}
\binom{d+h-i}{h+i+1}\binom{d+m}{i}.
\end{equation}
The constant coefficient is $u(m,d,0)=1$. Denote the summand by
\begin{equation}\label{eq:v-def}
v_h(m,d,i)=\frac1{d-i}\binom{h+i-1}{h}
\binom{d+h-i}{h+i+1}\binom{d+m}{i}.
\end{equation}
We first record the corank-one case.

\begin{lem}\label{lem:m-one-mode}
Let $M\in\mathrm{SP}(d,d+1)$ be loopless of positive rank. Then
$\lambda\leq1$, and $P_M(t)$ is strictly unimodal. Its mode is
$\lfloor d/3\rfloor$ for $\lambda=0$ and $\lfloor(d-1)/3\rfloor$ for
$\lambda=1$.
\end{lem}
\begin{proof}
For $d=1$, looplessness gives $\lambda=0$ and $P_M(t)=1$.
For $d\geq2$, use \eqref{eq:lambda-range} and \eqref{eq:sp-kl-formula}:
the two possible polynomials are $P_{U_{1,d}}$ and $P_{U_{1,d-1}}$.
It is therefore enough to treat the uniform case. Its coefficients satisfy
\[
u(1,d,i)=\frac1{d-i}\binom{d+1}{i}\binom{d-i}{i+1},
\qquad0\leq i\leq\left\lfloor\frac{d-1}{2}\right\rfloor.
\]
For $0\leq i<\lfloor(d-1)/2\rfloor$, the consecutive ratios are
\[
R_i(d)=\frac{(d+1-i)(d-2i-1)(d-2i-2)}
{(d-i-1)(i+1)(i+2)}.
\]
For fixed $i$, the ratio is strictly increasing with $d$ on $d\geq2i+3$.
Indeed, $(d-2i-2)/(d-i-1)=1-(i+1)/(d-i-1)$ is positive and increasing,
as are the other factors that depend on $d$. For $i\geq1$, the two
boundary values are
\[
R_i(3i+2)=\frac{i(2i+3)}{(i+2)(2i+1)}<1,
\qquad R_i(3i+3)=\frac{i+2}{i+1}>1.
\]
For $i=0$, the admissible ratios are $(d+1)(d-2)/2>1$.
Thus the coefficients increase exactly while $i<\lfloor d/3\rfloor$
and decrease thereafter. Replacing $d$ by $d-1$ gives the mode
$\lfloor(d-1)/3\rfloor$ when $\lambda=1$.
\end{proof}

We next derive the three recurrences for the uniform coefficients.
They can be obtained by creative telescoping using the package of
Koutschan~\cite{KoutschanHF}; we give
the finite telescoping identities so that the proof can be checked directly.

\begin{lem}
\label{lem:u-inhomogeneous-recurrences}
For integers $m\geq 1$, $d\geq 1$, and
$1\leq i\leq \lfloor(d-1)/2\rfloor$, the following recurrences in $m$
and $d$ hold:
\begin{equation}\label{eq:u-shift-m}
\begin{aligned}
u(m+1,d,i)
&=
\frac{d+m+1}{d-i+m+1}\,u(m,d,i)  \\
&\quad
+
\frac{1}{d-i}
\binom{i+m-1}{m}
\binom{d+m+1}{i}
\binom{d-i+m}{i+m+1}.
\end{aligned}
\end{equation}
\begin{equation}\label{eq:u-shift-d}
\begin{aligned}
u(m,d+1,i)
&=
\frac{d+m+1}{d-i+m+1}\,u(m,d,i)\\
&\quad
+
\frac{(d+m+1)!}
{(d-i)(d-i+1)(i+m)(i-1)!i!(m-1)!}\\
&\quad\times
\frac{1}{(d-2i)!(d-i+m+1)}.
\end{aligned}
\end{equation}
Moreover, for $1\leq i\leq \lfloor(d-3)/2\rfloor$, the following recurrence
in $i$ holds:
\begin{equation}\label{eq:u-shift-i}
\begin{aligned}
u(m,d,i+1)
&=
\frac{i-d-m}{i+1}\,u(m,d,i)\\
&\quad
+
\frac{m\bigl(d^2-3di-d+3i^2+im+2i\bigr)}
{i(i+1)(d-i-1)(d-i)}\\
&\quad\times
\binom{i+m-1}{m}
\binom{d+m}{i}
\binom{d-i+m}{i+m+1}.
\end{aligned}
\end{equation}
\end{lem}
\begin{proof}
Put $n=m+d$. The identity
\[
v_h(m+1,d,i)=\frac{n+1}{n-i+1}v_h(m,d,i)
\]
and the new summand at $h=m$ immediately give \eqref{eq:u-shift-m}.

For the other two shifts, abbreviate $v_h=v_h(m,d,i)$ and put
\[
\ell(h;d,i)=d^2-3di-d+3i^2+ih+2i.
\]
Define
\begin{align*}
Z_h&=\frac{n+1}{n-i+1}
\frac{h(h+i+1)}{(d-2i)(d-i+1)}v_h,\\
Y_h&=\frac{h(n-i)\ell(h;d,i)}
{i(i+1)(d-i-1)(d+h-i)}v_h.
\end{align*}
Using
\[
\frac{v_{h+1}}{v_h}
=\frac{(h+i)(d+h-i+1)}{(h+1)(h+i+2)},
\]
we obtain the two identities
\begin{align*}
v_h(m,d+1,i)-\frac{n+1}{n-i+1}v_h
&=Z_{h+1}-Z_h,\\
v_h(m,d,i+1)-\frac{i-n}{i+1}v_h
&=Y_{h+1}-Y_h.
\end{align*}
Sum over $h=0,\ldots,m-1$. Since $Z_0=Y_0=0$, the right-hand sides
are $Z_m$ and $Y_m$, respectively. Simplifying these boundary terms gives
exactly the inhomogeneous terms in \eqref{eq:u-shift-d} and
\eqref{eq:u-shift-i}. On the respective ranges of the lemma, $v_h>0$ and all the denominators
in these identities are nonzero.
\end{proof}

To compute the coefficient at index $\lfloor(d-1)/2\rfloor$, we use the
following formula, whose summation range is independent of $m$.
\begin{lem}\label{lem:u-fixed-d}
For $m\geq1$ and $1\leq i\leq\lfloor(d-1)/2\rfloor$,
\begin{equation}\label{eq:u-fixed-d}
u(m,d,i)=\frac{\binom{d+m}{i}}{(i+m)(i-1)!(m-1)!}
\sum_{h=2i}^{d-1}\frac{(h-i+m)!}{(h-i)(h-i+1)(h-2i)!}.
\end{equation}
\end{lem}
\begin{proof}
Fix $m,i$ and put $U_d=u(m,d,i)/\binom{d+m}{i}$.
Dividing \eqref{eq:u-shift-d} by $\binom{d+m+1}{i}$ gives
\[
U_{d+1}-U_d=
\frac{(d-i+m)!}{(i+m)(i-1)!(m-1)!(d-i)(d-i+1)(d-2i)!}.
\]
At $d=2i+1$, the defining sum and the identity
$\sum_{h=0}^{m-1}\binom{h+i-1}{h}=\binom{m+i-1}{m-1}$ give
\[
U_{2i+1}=\frac1{i+1}\binom{m+i-1}{m-1}
=\frac1{(i+m)(i-1)!(m-1)!}\frac{(i+m)!}{i(i+1)}.
\]
This is the $h=2i$ summand in the claimed formula. Summing the differences
from $2i+1$ to $d-1$ supplies all the remaining terms.
\end{proof}

\subsection{The sign of the correction difference}
We write $\Delta_i f(d,i)=f(d,i+1)-f(d,i)$ and
$\Delta_d f(d,i)=f(d+1,i)-f(d,i)$.
For $d\geq3$ and $1\leq i\leq\lfloor(d-1)/2\rfloor$, the
corank-one coefficient formula gives
\begin{equation}\label{eq:delta-d-u1}
\omega(d,i)=\Delta_d u(1,d-1,i)
=\frac{2d!}{(d-i-1)(d-i+1)(i-1)!(i+1)!(d-2i-1)!}>0.
\end{equation}
For $d\geq5$ and $1\leq i\leq\lfloor(d-3)/2\rfloor$, put
\begin{equation}\label{eq:delta-def}
\delta(d,i)=\omega(d,i+1)-\omega(d,i).
\end{equation}

\begin{lem}\label{lem:delta-sign}
For $d\geq5$ and $1\leq i\leq\lfloor(d-3)/2\rfloor$,
$\delta(d,i)$ is positive if and only if $d\geq3i+2$.
It is negative if and only if $d\leq3i+1$.
In particular, $\delta(d,i)$ never vanishes.
\end{lem}
\begin{proof}
Put $j=d-i$. Formula~\eqref{eq:delta-d-u1} gives
\begin{equation}\label{eq:omega-ratio}
\frac{\omega(d,i+1)}{\omega(d,i)}
=\frac{j^2-1}{i(i+2)}
\left(1-\frac{i+1}{j}\right)
\left(1-\frac{i}{j-2}\right).
\end{equation}
For fixed $i$, each factor is positive and strictly increasing for $j\geq i+3$.
At $d=3i+2$, the ratio equals
\[
\frac{(2i+1)(2i+3)}{4i(i+2)}=1+\frac3{4i(i+2)}>1.
\]
The admissible range $2i+3\leq d\leq3i+1$ is nonempty only when
$i\geq2$. At its upper endpoint, the ratio is
\[
\frac{4i(i^2-1)}{(i+2)(4i^2-1)}
=1-\frac{8i^2+3i-2}{(i+2)(4i^2-1)}<1.
\]
Monotonicity and the integrality of $d$ give the stated signs.
\end{proof}

For the rational recurrences below, it is useful to record
\begin{equation}\label{eq:delta-quartic}
\delta(d,i)=\frac{2d!\,Q(d,i)}
{i!(i+2)!(d-2i-1)!(d-i-2)(d-i-1)(d-i)(d-i+1)},
\end{equation}
where the quartic has the compact form
\begin{equation}\label{eq:quartic-Q}
\begin{aligned}
Q(d,i)={}&(d-i-1)(d-i+1)(d-2i-1)(d-2i-2)\\
&-i(i+2)(d-i-2)(d-i).
\end{aligned}
\end{equation}
These identities follow by subtracting $1$ from the ratio in
\eqref{eq:omega-ratio} and multiplying by $\omega(d,i)$.
All denominator factors are positive on the domain of $\delta$.

\subsection{Normalized differences}\label{subsec:b-rec}
For $m\geq1$, $d\geq5$, and $1\leq i\leq\lfloor(d-3)/2\rfloor$,
define
\begin{equation}\label{eq:b-def}
b(m,d,i)=\frac{\Delta_i u(m,d,i)}
{\lambda_{\max}(m,d)\delta(d,i)}.
\end{equation}
Taking consecutive differences in \eqref{eq:a-coeff} gives
\begin{equation}\label{eq:b-factor}
\Delta_i a(m,d,i)=\delta(d,i)
\bigl(\lambda_{\max}(m,d)b(m,d,i)-\lambda\bigr).
\end{equation}
Since $\lambda\leq\lambda_{\max}$, the inequality $b>1$ implies that
$\Delta_i a(m,d,i)$ has the same sign as $\delta(d,i)$.
The sharper bound in \eqref{eq:lambda-range} shows that the weaker comparison
\begin{equation}\label{eq:sharp-barrier}
b(m,d,i)>\theta_d(m),\qquad
\theta_d(m)=\frac{m+d+2}{2(\max\{m,d\}+1)}
\end{equation}
also suffices. Notice that $0<\theta_d(m)\leq1$.
To obtain recurrences for $b$, we eliminate $u(m,d,i)$ from the uniform
coefficient recurrences.

\begin{prop}\label{prop:b-rec}
Let $m\geq1$, $d\geq5$, and
$1\leq i\leq\lfloor(d-3)/2\rfloor$. Then
\begin{align}
b(m,d+1,i)&=Q_1(m,d,i)b(m,d,i)+R_1(m,d,i),
\label{eq:b-rec-d}\\
b(m+1,d,i)&=Q_3(m,d,i)b(m,d,i)+R_3(m,d,i).
\label{eq:b-rec-m}
\end{align}
If also $d\geq7$ and $1\leq i\leq\lfloor(d-5)/2\rfloor$, then
\begin{equation}\label{eq:b-rec-i}
b(m,d,i+1)=Q_2(m,d,i)b(m,d,i)+R_2(m,d,i).
\end{equation}
The rational functions $Q_j,R_j$ are given in
Appendix~\ref{app:b-rec-data}. On its stated domain, $Q_1(m,d,i)$ is
negative if and only if $d=3i+1$, and is positive otherwise.
On the domain of \eqref{eq:b-rec-i}, $Q_2(m,d,i)$ is positive if and only if
$3i+2\leq d\leq3i+4$, and is negative otherwise. Finally,
\begin{equation}\label{eq:Q3-positive}
Q_3(m,d,i)=\frac{(m+1)(d+m+3)}{(d+m+1)(d-i+m+1)}>0.
\end{equation}
\end{prop}
\begin{proof}
Write $n=m+d$, $D(m,d,i)=\Delta_i u(m,d,i)$, and let $V(m,d,i)$ be
the inhomogeneous term in \eqref{eq:u-shift-i}. Thus
\begin{equation}\label{eq:D-elimination}
D(m,d,i)=-\frac{n+1}{i+1}u(m,d,i)+V(m,d,i).
\end{equation}
Let $E_m$ and $E_d$ denote the inhomogeneous terms in
\eqref{eq:u-shift-m} and \eqref{eq:u-shift-d}, respectively.
Both recurrences have coefficient $(n+1)/(n-i+1)$.
Eliminating $u$ by \eqref{eq:D-elimination} gives, for either the $m$-shift
or the $d$-shift,
\begin{equation}\label{eq:D-shift}
\begin{aligned}
D_{\mathrm{new}}
&=\frac{n+2}{n-i+1}D+V_{\mathrm{new}}
-\frac{n+2}{n-i+1}V-\frac{n+2}{i+1}E,
\end{aligned}
\end{equation}
where $E=E_m$ or $E_d$, as appropriate.
For the shift in $i$, the same elimination gives
\begin{equation}\label{eq:D-shift-i}
D(m,d,i+1)=-\frac{n-i}{i+2}D(m,d,i)
+V(m,d,i+1)-\frac{i+1}{i+2}V(m,d,i).
\end{equation}
Divide these identities by the shifted values of
$\lambda_{\max}\delta$. Cancellation of factorials yields the expressions
in Appendix~\ref{app:b-rec-data} and hence all three recurrences.

Every shifted index stays in the domain of $b$. Lemma~\ref{lem:delta-sign}
shows that the quartic factors are nonzero, while all remaining denominator
factors are positive. Thus each division is valid.
The quartic $Q(d,i)$ has the same sign as $\delta(d,i)$ by
\eqref{eq:delta-quartic}. In \eqref{eq:Q1-compact} and \eqref{eq:Q2-compact},
all factors other than the quartics and the explicit minus sign are positive.
By Lemma~\ref{lem:delta-sign}, $Q(d,i)$ and $Q(d+1,i)$ have opposite signs
exactly when $d=3i+1$, whereas $Q(d,i)$ and $Q(d,i+1)$ have opposite signs
exactly when $3i+2\leq d\leq3i+4$. These comparisons give the stated signs
of $Q_1$ and $Q_2$. Finally, \eqref{eq:Q3-positive} gives $Q_3>0$.
\end{proof}

\section{Proof of unimodality}\label{sec:unimodality}
Throughout this section, $m\geq2$ and $d\geq5$.
We will determine the signs of all consecutive differences except the one
at $i=(d-2)/3$ when $d\equiv2\pmod3$. By \eqref{eq:b-factor}, away from
$d=3i+2$ the required signs follow mainly from $b>1$.
On $d=3i+1$ we compare $b$ with $\theta_d(m)$, and the case
$m=2$, $d=3i+3$ is handled directly in the proof of
Theorem~\ref{thm:first-difference-sign}. The line $d=3i+2$ is treated
in Section~\ref{sec:mode}.

\subsection{Coranks two and three}
For $m=2,3$, the coefficient formula is a sum of only two or three terms.
We use the resulting rational expressions to start the induction on $m$.

\begin{lem}\label{lem:b-base-m2}
For $d\geq5$ and $1\leq i\leq\lfloor(d-3)/2\rfloor$, we have
$b(2,d,i)>1$ unless $d=3i+2$ or $d=3i+3$. On these two lines,
\[
 b(2,3i+2,i)<0,\qquad 0<b(2,3i+3,i)<1.
\]
\end{lem}
\begin{proof}
Write
\[
 b(2,d,i)-1=\frac{N_2(d,i)}
 {2(i+1)(i+3)(d-i+2)Q(d,i)}.
\]
For $d\geq3i+4$, put $x=i-1$ and $y=d-3i-4$.
The polynomial $N_2(3x+y+7,x+1)$ has nonnegative coefficients and a
positive constant term; its coefficient data are summarized in
Appendix~\ref{app:positivity-certificates}. For $2i+3\leq d\leq3i+1$,
put $x=3i-d+1$ and $y=d-2i-2$, so that
\[
    i=x+y+1,\qquad d=2x+3y+4.
\]
Here $x,y\geq0$, and $-N_2(2x+3y+4,x+y+1)$ also has nonnegative
coefficients and positive constant term. Lemma~\ref{lem:delta-sign} gives
$Q>0$ in the first range and $Q<0$ in the second. Thus $b(2,d,i)-1>0$
in both ranges.

On the line $d=3i+2$, direct factorization gives
\[
 b(2,3i+2,i)
 =
 -\frac{2i^3+4i^2+3i+3}{(i+1)(i+3)}<0.
\]
On the other exceptional line, $d=3i+3$, substitution gives
\[
b\!\left(2,d,\frac{d-3}{3}\right)=
\frac{3(d+4)(d+9)(2d-3)(2d+3)}{(d+6)(2d+9)(8d^2+3d-9)},
\]
and
\[
1-b\!\left(2,d,\frac{d-3}{3}\right)
=
\frac{2(d+3)(2d^3+3d^2+27d+81)}
{(d+6)(2d+9)(8d^2+3d-9)}>0.
\]
\end{proof}

\begin{lem}\label{lem:b-base-m3}
For $d\geq5$ and $1\leq i\leq\lfloor(d-3)/2\rfloor$, we have
$b(3,d,i)>1$ unless $d=3i+2$. On that line, $b(3,3i+2,i)<0$.
\end{lem}
\begin{proof}
The same calculation gives
\[
 b(3,d,i)-1=
\frac{N_3(d,i)}
{4(i+1)(i+3)(i+4)(d-i+2)(d-i+3)Q(d,i)}.
\]
For $d\geq3i+3$, substitute $i=x+1$ and $d=3x+y+6$.
For $2i+3\leq d\leq3i+1$, use $i=x+y+1$ and $d=2x+3y+4$.
In both cases the new variables are nonnegative. The respective
polynomials $N_3$ and $-N_3$ have nonnegative coefficients and positive
constant term, as recorded in Appendix~\ref{app:positivity-certificates}.
The signs of $Q$ then give $b(3,d,i)>1$.

On the remaining line $d=3i+2$,
\[
b(3,3i+2,i)
=-\frac{(3i+7)(2i^4+5i^3+4i^2-2i+6)}
{(i+1)(i+3)(i+4)(2i+5)}.
\]
For $i=z+1$, the quartic factor in the numerator becomes
\[
2z^4+13z^3+31z^2+29z+15,
\]
which is positive for $z\geq0$.  Hence $b(3,3i+2,i)<0$, as
required.
\end{proof}

\subsection{Induction on the corank}
To extend these estimates to larger coranks, we use the elementary identity
\[
x'-1=Q(x-1)+(Q+R-1),\qquad x'=Qx+R.
\]
In particular, $x>1$ implies $x'>1$ whenever $Q>0$ and $Q+R>1$.
The next lemma supplies the latter inequality for the recurrence in $m$.

\begin{lem}\label{lem:m-step}
The inequality $Q_3(m,d,i)+R_3(m,d,i)>1$ holds in the following ranges:
\begin{align*}
&m\geq3,\quad i\geq2,\quad d\geq3i+3;\\
&m\geq3,\quad i\geq3,\quad 2i+3\leq d\leq3i;\\
&i\geq1,\quad d=3i+2,\quad m\geq i+3.
\end{align*}
\end{lem}
\begin{proof}
For the first two ranges, write
\begin{equation}\label{eq:m-step-numerator}
Q_3+R_3-1=
\frac{N(m,d,i)}{4(i+m)(i+m+1)(i+m+2)(m+d+1)(m+d-i+1)Q(d,i)}.
\end{equation}
The compact recurrence formulas in Appendix~\ref{app:b-rec-data} define
$N$ explicitly and show that it is a polynomial.
In the first range, substitute
\[
m=a+3,\qquad i=b+2,\qquad d=3b+9+c.
\]
In the second range, substitute
\[
m=a+3,\qquad i=b+c+3,\qquad d=2b+3c+9.
\]
In both cases $a,b,c\geq0$. The corresponding polynomials $N$ and $-N$
have nonnegative coefficients and a positive constant term, as recorded in
Appendix~\ref{app:positivity-certificates}. The sign of $Q$ is positive
in the first range and negative in the second by
Lemma~\ref{lem:delta-sign}. Thus $Q_3+R_3>1$ in both ranges.

For the third range, put $x=m-i-3\geq0$. Evaluating $Q_3,R_3$ at
$(m,3i+2,i)$ and simplifying \eqref{eq:m-step-numerator}, we obtain
\begin{equation}\label{eq:critical-m-step}
Q_3+R_3-1=\frac{\sum_{j=0}^4 c_j(i)x^j}
{3(2i+x+3)(2i+x+4)(2i+x+5)(3i+x+6)(4i+x+6)},
\end{equation}
where
\begin{align*}
c_0(i)&=4(i+2)^2(2i+3)(10i^3+35i^2+13i-3),\\
c_1(i)&=(i+2)(32i^5+420i^4+1396i^3+1519i^2+301i-78),\\
c_2(i)&=48i^5+432i^4+1204i^3+1154i^2+159i-54,\\
c_3(i)&=(i+2)(18i^3+71i^2+8i-3),\\
c_4(i)&=i^2(2i+7).
\end{align*}
$c_j(i)>0$ for $i\geq1$ and $0\leq j\leq4$, so $Q_3+R_3>1$ here as well.
\end{proof}

\begin{lem}\label{lem:b-inductive}
For $m\geq3$ and $i\geq2$, one has $b(m,d,i)>1$ whenever
\[
d\geq3i+3\qquad\text{or}\qquad 2i+3\leq d\leq3i.
\]
\end{lem}
\begin{proof}
Fix $(d,i)$. Lemma~\ref{lem:b-base-m3} gives the assertion at $m=3$.
Since $Q_3>0$, the recurrence \eqref{eq:b-rec-m} and
Lemma~\ref{lem:m-step} preserve $b>1$ as $m$ increases.
\end{proof}

\subsection{The boundary \texorpdfstring{$d=3i+1$}{d=3i+1}}
The estimate $b>1$ is unnecessarily strong on this line. We use the sharper
bound $\theta_d$ from \eqref{eq:sharp-barrier} instead.

\begin{lem}\label{lem:boundary-pairs}
For $i\geq2$, $d=3i+1$, and $m\geq2$, one has
\[
b(m,d,i)>\theta_d(m).
\]
\end{lem}
\begin{proof}
For $m=2$, Lemma~\ref{lem:b-base-m2} gives $b(2,d,i)>1$, whereas
$\theta_d(2)=(d+4)/(2d+2)<1$. Define the two rational branches
\[
L_d(m)=\frac{m+d+2}{2(d+1)},\qquad
T_d(m)=\frac{m+d+2}{2(m+1)}.
\]
They agree at $m=d$. The algebraic inequalities in
Lemma~\ref{lem:boundary-affine} are
\begin{align*}
Q_3L_d(m)+R_3&>L_d(m+1) &&(2\leq m\leq d-1),\\
Q_3T_d(m)+R_3&>T_d(m+1) &&(m\geq d).
\end{align*}
Because $Q_3>0$, the first inequality propagates $b>L_d$ from $m=2$ to
$m=d$, and the second propagates $b>T_d$ from there. These are exactly
the two branches of $\theta_d$, which proves the claim.
\end{proof}

\subsection{The first two nonconstant coefficients}\label{subsec:boundary-i1}
The comparison at $i=1$ admits a shorter argument directly from the
coefficient formula. The following proposition covers every rank needed
in this section.

\begin{prop}\label{prop:i1-large}
For $m\geq2$ and $d\geq6$, every $M\in\mathrm{SP}(d,m+d)$ satisfies
\[
a(m,d,2)>a(m,d,1).
\]
\end{prop}
\begin{proof}
For $m=2$, $d\geq7$, and for $m=3$, $d\geq6$,
Lemmas~\ref{lem:b-base-m2} and~\ref{lem:b-base-m3} give $b(m,d,1)>1$.
Since $\delta(d,1)>0$, formula~\eqref{eq:b-factor} proves the claim in
these cases.

Now let $m\geq4$. Put $C=\binom{m+d}{m}$ and
$\eta=d^3-6d^2+5d+3$, so that $\delta(d,1)=\eta/3>0$.
The coefficient formula at $i=1$ is
\[
u(m,d,1)=\frac{d}{m+1}C-(m+d).
\]
Substituting this into \eqref{eq:u-shift-i} with $i=1$ gives
\[
\begin{aligned}
\frac{u(m,d,2)-u(m,d,1)}{\delta(d,1)}
&=\frac{3(m+d)(m+d+1)}{2\eta}\\
&\quad+\frac{3d\bigl((d^2-5d+2)m-2(d+1)\bigr)}
{2\eta(m+1)(m+2)}C.
\end{aligned}
\]
The first term is positive. Subtracting $1/(m+1)$ from the coefficient
of $C$ in the second term gives
\[
\frac{N(m,d)}{2\eta(m+1)(m+2)},
\]
where
\[
N(m,d)=d^3m-4d^3-3d^2m+18d^2-4dm-26d-6m-12.
\]
For $m\geq5$, $d\geq6$, write $m=a+5$ and $d=c+6$. Then
\[
N(m,d)=a(c^3+15c^2+68c+78)+c^3+21c^2+98c+6>0.
\]
For $m=4$, $d\geq8$, we instead have
\[
N(4,d)=6(d^2-7d-6)=6\bigl((d-8)^2+9(d-8)+2\bigr)>0.
\]
Thus the ratio above exceeds $C/(m+1)\geq\lambda$ in both ranges.
Equation~\eqref{eq:a-coeff} proves the desired comparison.

Only $(m,d)=(2,6),(4,6),(4,7)$ remain. Using
$\lambda\leq C/(d+1)$, the respective lower bounds for
$a(m,d,2)-a(m,d,1)$ are
\[
6,\qquad103,\qquad\frac{2871}{4}.
\]
All are positive, which completes the proof.
\end{proof}

\begin{lem}\label{lem:a1-positive}
Let $M\in \mathrm{SP}(d,m+d)$ with $m\geq 2$ and $d\geq 5$.
Then
\[
a(m,d,1)>1=a(m,d,0).
\]
\end{lem}
\begin{proof}
Put $n=m+d$ and $C=\binom{n}{m}$. The coefficient formula gives
\[
a(m,d,1)=\frac{dC}{m+1}-n-\lambda(d-1).
\]
For $m=2$, the bound $\lambda\leq C/(d+1)=(d+2)/2$ yields
\[
a(2,d,1)\geq\frac{(d+2)(d-3)(d+1)}6\geq14.
\]

Now let $m\geq3$. Using the weaker bound $\lambda\leq C/(m+1)$ gives
\[
a(m,d,1)\geq\frac{C}{m+1}-n
=\frac1d\binom{n}{d-1}-n.
\]
Here $n\geq8$ and $4\leq d-1\leq n-4$. The symmetry and unimodality
of a row of binomial coefficients imply
$\binom{n}{d-1}\geq\binom{n}{4}$; this also follows immediately from
$\binom{n}{j+1}/\binom{n}{j}=(n-j)/(j+1)$.
Since $d\leq n-3$, we obtain
\[
a(m,d,1)\geq\frac1{n-3}\binom{n}{4}-n
=n\left(\frac{(n-1)(n-2)}{24}-1\right)
\geq\frac{3n}{4}\geq6.
\]
Thus $a(m,d,1)>1$ in both cases.
\end{proof}

\subsection{The signs of all noncritical differences}
We now combine the estimates for $b$ with the sign of $\delta$.
The two ranges in the next theorem cover every nonconstant consecutive
comparison except $i=(d-2)/3$ when $d\equiv2\pmod3$.

\begin{theo}\label{thm:first-difference-sign}
Let $M\in\mathrm{SP}(d,m+d)$, with $m\geq2$ and $d\geq5$. Then
\begin{align*}
a(m,d,i+1)&>a(m,d,i)
&&\left(1\leq i\leq\left\lfloor\frac{d-3}{3}\right\rfloor\right),\\
a(m,d,i+1)&<a(m,d,i)
&&\left(\left\lceil\frac{d-1}{3}\right\rceil
\leq i\leq\left\lfloor\frac{d-1}{2}\right\rfloor-1\right).
\end{align*}
\end{theo}
\begin{proof}
The comparison at $i=1$ occurs only for $d\geq6$, and is precisely
Proposition~\ref{prop:i1-large}.

Now let $i\geq2$. On the increasing side, $d\geq3i+3$, so $\delta>0$.
For $m\geq3$, Lemma~\ref{lem:b-inductive} gives $b>1$ and hence the
desired sign. For $m=2$, the same argument uses
Lemma~\ref{lem:b-base-m2}, except when $d=3i+3$.
On that line,
\[
\lambda_{\max}(2,d)b(2,d,i)-\frac{d+2}{2}
=\frac{d(d+2)(8d^3+66d^2-315d-513)}
{2(d+6)(2d+9)(8d^2+3d-9)}>0.
\]
The cubic is $8x^3+282x^2+2817x+7830$ for $d=x+9$.
Since $\lambda\leq(d+2)/2$, formula~\eqref{eq:b-factor} again gives
$a(m,d,i+1)>a(m,d,i)$.

On the decreasing side, $2i+3\leq d\leq3i+1$ and $\delta<0$.
If $d\leq3i$, use $b>1$ from Lemma~\ref{lem:b-inductive} for $m\geq3$
and Lemma~\ref{lem:b-base-m2} for $m=2$.
If $d=3i+1$, Lemma~\ref{lem:boundary-pairs} gives $b>\theta_d(m)$.
In either case $\lambda_{\max}b-\lambda>0$, which gives the negative
sign in \eqref{eq:b-factor}.
\end{proof}

Together with Lemma~\ref{lem:a1-positive}, this proves unimodality for
$m\geq2$ and $d\geq5$. For $d\not\equiv2\pmod3$, the unique mode is
$\lfloor d/3\rfloor$. For $d\equiv2\pmod3$, only the comparison between indices
$\lfloor d/3\rfloor$ and $\lfloor d/3\rfloor+1$ remains.
Either sign, as well as equality, gives a unimodal sequence.

\Needspace{20\baselineskip}
\section{The mode}\label{sec:mode}

Throughout this section, $M\in\mathrm{SP}(d,m+d)$ has $\lambda$
circuit-hyperplanes, with $m\geq2$ and $d\geq5$.
We determine which of the two possible indices is a mode, completing the
proof of Theorem~\ref{thm:mode}.

By Theorem~\ref{thm:first-difference-sign}, it remains to consider
$d=3i+2$. The sign of
\[
a(m,3i+2,i+1)-a(m,3i+2,i)
=\delta(3i+2,i)\bigl(\lambda_{\max}(m,3i+2)b(m,3i+2,i)-\lambda\bigr)
\]
decides the mode. Since $\delta(3i+2,i)>0$, we need only compare
$\lambda_{\max}b$ with $\lambda$. We first treat $i\geq2$; rank five
will be considered separately. We bound $b$ on the diagonal $m=i+2$.
The same bound will then give the positive comparisons by recurrence in $m$.

\subsection{The negative range}
\begin{lem}\label{lem:critical-sign}
For $i\geq2$ and $1\leq m\leq i+2$, one has
\[
b(m,3i+2,i)<0.
\]
\end{lem}
\begin{proof}
We first treat $m\leq i+1$. The corank-one case follows from
Lemma~\ref{lem:m-one-mode} and $\delta(3i+2,i)>0$.
On $d=3i+2$, the inhomogeneous term in the $m$-recurrence factors as
\[
R_3(m,3i+2,i)=
-\frac{(i+2)(2i+3)(m+1)(3i+m+5)P(m,i)}
{3(i+m)(i+m+1)(i+m+2)(2i+m+3)(3i+m+3)},
\]
where
\[
P(m,i)=3i^3-2mi^2+6i^2-m^2i-4mi+3i-m.
\]
For $1\leq m\leq i$, put $r=i-m\geq0$. Then
\[
P(m,m+r)=4m^2r+2m^2+7mr^2+8mr+2m+3r^3+6r^2+3r>0.
\]
Hence $R_3<0$ in every step from $m=1$ to $m=i+1$.
Since $Q_3>0$, induction in $m$ preserves negativity on this range.

For the remaining boundary, put $z(i)=b(i+2,3i+2,i)$.
Appendix~\ref{app:diagonal-recurrences} gives
\[
z(i+1)=\alpha_2(i)z(i)+\beta_2(i),\qquad \alpha_2(i)<0.
\]
We claim the stronger bound
\begin{equation}\label{eq:z-bounds}
-\frac38<z(i)<0\qquad(i\geq2).
\end{equation}
The initial value is $z(2)=-49/135$, with $z(2)+3/8=13/1080>0$.
For the induction step, we have
\begin{align}
\beta_2(i)-\frac38\alpha_2(i)
&=-\frac{(i+3)^2(2i+5)(50i^2+151i+105)}
{24(i+2)(2i+1)(2i+3)(3i+5)(3i+7)}<0,
\label{eq:z-upper-residual}\\
\beta_2(i)+\frac38
&=\frac{(4i+9)K(i)}
{24(i+2)(2i+1)(3i+5)(3i+7)(4i+5)}>0,
\label{eq:z-lower-residual}
\end{align}
where
\[
\begin{aligned}
K(i)&=106 i^4+377 i^3+32 i^2-965 i-750\\
&=(106i^4-750)+i(377i^2-965)+32i^2>0
\qquad(i\geq2).
\end{aligned}
\]
Since $\alpha_2(i)<0$, the inductive hypothesis gives
\[
-\frac38<\beta_2(i)<z(i+1)
<\beta_2(i)-\frac38\alpha_2(i)<0.
\]
This proves \eqref{eq:z-bounds} and hence the remaining case $m=i+2$.
\end{proof}

\subsection{The positive range}
We apply the corank recurrence first at $m=i+2$ and then at $m=i+3$.
The first step uses $z(i)>-3/8$ from \eqref{eq:z-bounds} to obtain a lower
bound for $b(i+3,3i+2,i)$; the second uses that bound to prove
$b(i+4,3i+2,i)>1$.
\begin{lem}\label{lem:positive-diagonal}
One has
\begin{align*}
b(i+3,3i+2,i)&>\theta_{3i+2}(i+3) &&(i\geq4),\\
b(m,3i+2,i)&>1 &&(i\geq2,\ m\geq i+4).
\end{align*}
\end{lem}
\begin{proof}
With $z(i)=b(i+2,3i+2,i)$, the recurrence \eqref{eq:b-rec-m} gives
\begin{equation}\label{eq:critical-corank-bridge}
b(i+3,3i+2,i)
=\frac{(i+3)(4i+7)}{(3i+5)(4i+5)}
\left(z(i)+\frac{3i^2+5i+1}{6(i+1)}\right).
\end{equation}
The coefficient of $z(i)$ in \eqref{eq:critical-corank-bridge} is positive.
We substitute the lower bound $z(i)>-3/8$ and use the identity
\[
-\frac38+\frac{3i^2+5i+1}{6(i+1)}
=\frac{(3i-1)(4i+5)}{24(i+1)}.
\]
The factor $4i+5$ then cancels with the denominator in
\eqref{eq:critical-corank-bridge}, giving
\begin{equation}\label{eq:critical-corank-lower}
b(i+3,3i+2,i)>L(i):=\frac{(i+3)(3i-1)(4i+7)}
{24(i+1)(3i+5)}.
\end{equation}
At this point $\theta_{3i+2}(i+3)=(4i+7)/(6(i+1))$, so
\begin{equation}\label{eq:critical-corank-sharp}
L(i)-\theta_{3i+2}(i+3)
=\frac{(4i+7)(3i^2-4i-23)}{24(i+1)(3i+5)}>0
\qquad(i\geq4).
\end{equation}
This proves the first inequality.

Apply \eqref{eq:b-rec-m} once more. Since $Q_3>0$, we obtain
\begin{equation}\label{eq:critical-corank-margin}
\begin{aligned}
b(i+4,3i+2,i)-1
&>Q_3(i+3,3i+2,i)L(i)+R_3(i+3,3i+2,i)-1\\
&=\frac{H(i)}{12(i+1)(2i+3)(2i+5)(3i+5)},
\end{aligned}
\end{equation}
where
\[
\begin{aligned}
H(i)&=48 i^5+286 i^4+391 i^3-502 i^2-1543 i-960\\
&=(48i^5-960)+i^2(286i^2-502)+i(391i^2-1543)>0
\qquad(i\geq2).
\end{aligned}
\]
Each parenthesized expression is positive for $i\geq2$.
Thus $b(i+4,3i+2,i)>1$. Lemma~\ref{lem:m-step} propagates this
inequality to every $m\geq i+4$.
\end{proof}

Together with Lemma~\ref{lem:critical-sign}, this settles all comparisons
with $i\geq2$ except $i=2,3$ at $m=i+3$, namely $(m,d)=(5,8),(6,11)$.

\subsection{Rank five and the remaining exceptions}
For $d=5$, the coefficient formula gives
\begin{align*}
a(m,5,1)&=\frac{m(m+5)(m^2+9m+26)}{24}-4\lambda,\\
a(m,5,2)&=\frac{m(m+1)(m+4)(m+5)}{12}-5\lambda.
\end{align*}
Hence
\begin{equation}\label{eq:rank-five-difference}
a(m,5,2)-a(m,5,1)=T(m)-\lambda,
\qquad T(m)=\frac{m(m+5)(m^2+m-18)}{24}.
\end{equation}
Since $T(2)=-7$ and $T(3)=-6$, the difference is negative for $m=2,3$.
For $m\geq5$, put
\[
B(m)=\frac1{m+1}\binom{m+5}{m},
\]
which bounds $\lambda$ by \eqref{eq:lambda-range}. The factorization
\[
T(m)-B(m)=\frac{(m+5)(m-6)(m^2+5m+1)}{30}
\]
shows that the difference is positive for $m\geq7$.
The intervening values are
\[
T(4)=3,\qquad T(5)=25,\qquad T(6)=B(6)=66.
\]
Thus the mode shifts from $2$ to $1$ at $\lambda=4$ for $m=4$, and at
$\lambda=26$ for $m=5$. At $m=6$, equality is possible only when
$\lambda=66$, and the difference is otherwise positive.
The three plateaux have polynomials
\begin{equation}\label{eq:plateau-polynomials}
1+105t+105t^2,\qquad1+100t+100t^2,\qquad1+55t+55t^2.
\end{equation}
Their existence will be verified below.

For the two parameter pairs $(m,d)=(5,8),(6,11)$ left above,
the same coefficient formula gives
\begin{align}
a(5,8,3)-a(5,8,2)&=494-6\lambda,
\label{eq:rank-eight-difference}\\
a(6,11,4)-a(6,11,3)&=51680-55\lambda.
\label{eq:rank-eleven-difference}
\end{align}
The zeros of these affine expressions are $247/3$ and $10336/11$,
respectively. Neither is an integer. The differences are therefore negative
exactly when $\lambda\geq83$ and $\lambda\geq940$, respectively, and
no additional plateau occurs.

\begin{proof}[Proof of Theorem~\ref{thm:mode}]
For $d\not\equiv2\pmod3$, Theorem~\ref{thm:first-difference-sign} and
Lemma~\ref{lem:a1-positive} give the unique mode $\lfloor d/3\rfloor$.
Suppose $d=3i+2$.
For $i\geq2$, Lemma~\ref{lem:critical-sign} puts the mode at $i$ when
$m\leq i+2$. For $i\geq2$ and $m\geq i+4$,
Lemma~\ref{lem:positive-diagonal} puts the mode at $i+1$.
The same conclusion holds at $m=i+3$ for $i\geq4$.
The two remaining cases, $i=2,3$ at $m=i+3$, are exactly
$(m,d)=(5,8),(6,11)$. Their thresholds follow from
\eqref{eq:rank-eight-difference} and \eqref{eq:rank-eleven-difference}.
Finally, \eqref{eq:rank-five-difference} settles $i=1$.
Since $m\geq i+3$ is equivalent to $d\leq3m-7$ on $d=3i+2$, these
comparisons give exactly $\kappa(m,d)$ and the four exceptions stated.
\end{proof}

\begin{proof}[Proof of Theorem~\ref{thm:strict-unimodality}]
If $M$ has rank zero or corank zero, then $P_M(t)=1$.
For $1\leq d\leq4$, the coefficient sequence has at most two terms, so
it is unimodal. We may therefore assume $d\geq5$ and $m\geq1$.
The case $m=1$ follows from Lemma~\ref{lem:m-one-mode}.
For $m\geq2$, Lemma~\ref{lem:a1-positive} gives
$a(m,d,0)=1<a(m,d,1)$.
If $d\not\equiv2\pmod3$, all the remaining differences have the strict
signs given by Theorem~\ref{thm:first-difference-sign}.

If $d=3i+2$, those signs give
\[
a(m,d,0)<\cdots<a(m,d,i),\qquad
 a(m,d,i+1)>a(m,d,i+2)>\cdots.
\]
The second chain is vacuous when $i+1=\lfloor(d-1)/2\rfloor$.
Either sign for the difference between indices $i$ and $i+1$ gives strict
unimodality; equality gives a plateau of length two. The preceding
comparisons show that equality occurs exactly at
\[
(m,d,\lambda)=(4,5,3),\quad(5,5,25),\quad(6,5,66).
\]
The rank-five calculation gives the three polynomials in
\eqref{eq:plateau-polynomials}, whose realizability is established in the
next subsection.
\end{proof}

\subsection{Realizability of the exceptional parameters}\label{subsec:realizability}
We distinguish the coefficient comparisons above from the existence of
matroids with the specified parameters. We recall, with a proof, the standard
description of sparse paving matroids by circuit-hyperplanes; see
Lee, Nasr and Radcliffe~\cite{LN2021} and
Pendavingh and van der Pol~\cite{PV2015}.

\begin{lem}\label{lem:stable-family-construction}
Let $1\leq r<n$, and let $\mathcal H$ be a family of $r$-subsets of an
$n$-element set $E$ such that
\[
|H\cap H'|\leq r-2\qquad(H,H'\in\mathcal H,\ H\ne H').
\]
Then $\binom Er\setminus\mathcal H$ is the set of bases of a rank-$r$
sparse paving matroid whose circuit-hyperplanes are exactly $\mathcal H$.
The matroid is loopless when $r\geq2$.
\end{lem}
\begin{proof}
Every $(r-1)$-subset of $E$ lies in at least two $r$-subsets, at most one
of which is in $\mathcal H$. It therefore extends to a member of
$\mathcal B=\binom Er\setminus\mathcal H$, and in particular
$\mathcal B\ne\varnothing$.
To check basis exchange, take $B_1,B_2\in\mathcal B$ and
$e\in B_1\setminus B_2$. If $B_2\setminus B_1$ has one element $f$,
then $(B_1\setminus\{e\})\cup\{f\}=B_2$. Otherwise there are at least
two candidate sets $(B_1\setminus\{e\})\cup\{f\}$, and any two share
$r-1$ elements. At most one candidate is in $\mathcal H$, so a basis
exchange is again possible.

Every $(r-1)$-subset is independent, hence the resulting matroid is
paving. Each $H\in\mathcal H$ is a circuit of rank $r-1$. For
$e\notin H$ and $h\in H$, the set $(H\setminus\{h\})\cup\{e\}$
is a basis, so $H$ is also a flat. Conversely, every circuit-hyperplane
has size $r$ and is a nonbasis, hence belongs to $\mathcal H$.
The complements of the members of $\mathcal H$ have pairwise intersections
of size at most $n-r-2$. Applying the same argument to the complementary
bases shows that the dual is paving. If $r\geq2$, every singleton lies in
an independent $(r-1)$-subset, so there are no loops.
\end{proof}

Let $A(n,4,w)$ be the largest size of a family of $w$-subsets of an
$n$-element set with pairwise intersections of size at most $w-2$.
Equivalently, this is the usual binary constant-weight code number with
minimum Hamming distance at least four. By Lemma~\ref{lem:stable-family-construction},
for $2\leq d<n$, the possible circuit-hyperplane counts at fixed $n,d$
are precisely
\[
0,1,\ldots,A(n,4,d).
\]
Indeed, every subfamily of an admissible family is again admissible.
Complementation gives $A(n,4,w)=A(n,4,n-w)$.

For the rank-five plateaux, take the Witt $4$-$(11,5,1)$
design~\cite{Witt1938}. It has $66$ blocks, and distinct blocks intersect
in at most three points. Each point lies in $30$ blocks and each pair in
$12$ blocks. Thus $36$ blocks avoid a fixed point, and $18$ avoid two fixed
points. On the remaining ten and nine points, respectively, these families
satisfy the hypotheses of Lemma~\ref{lem:stable-family-construction}.
Choosing $25$ blocks from the first family, $3$ from the second, and all
$66$ blocks on eleven points realizes the respective triples
$(5,5,25)$, $(4,5,3)$, and $(6,5,66)$.

The counts $18$ and $36$ are also optimal. Counting incidences between
points and blocks gives the elementary bound
\begin{equation}\label{eq:johnson-count}
A(n,4,w)\leq
\left\lfloor\frac{n}{w}A(n-1,4,w-1)\right\rfloor\qquad(w\geq2).
\end{equation}
In fact, the blocks containing a fixed point, with that point deleted,
form a family counted by $A(n-1,4,w-1)$.
Since a family of $2$-subsets has to be a matching,
$A(7,4,2)=3$. Successive applications of \eqref{eq:johnson-count} give
\[
A(8,4,3)\leq8,\qquad A(9,4,4)\leq18,\qquad A(10,4,5)\leq36.
\]
Together with the Witt constructions and complementation, this proves
\[
A(9,4,5)=18,\qquad A(10,4,5)=36.
\]
Consequently, all exceptional values $4\leq\lambda\leq18$ and
$26\leq\lambda\leq36$ occur for $(m,d)=(4,5)$ and $(5,5)$,
respectively, and no larger value occurs at either pair.

For $(m,d)=(5,8)$, an explicit code of size $123$ is given in
Brouwer's tables~\cite{BrouwerTables}. Taking complements and subfamilies
therefore realizes every $83\leq\lambda\leq123$.
This is a lower-bound construction, not a determination of $A(13,4,5)$.
For $(m,d)=(6,11)$, the final exceptional case in
Theorem~\ref{thm:mode} is realizable if and only if $A(17,4,6)\geq940$.
The same elementary counting bound gives
\[
A(14,4,3)\leq28,\quad A(15,4,4)\leq105,\quad
A(16,4,5)\leq336,\quad A(17,4,6)\leq952,
\]
starting with $A(13,4,2)=6$.
Thus only $940\leq\lambda\leq952$ could realize this exceptional case.
We do not assert its realizability; the coefficient comparison remains
valid for every matroid with these parameters.

\begin{coro}\label{cor:asymptotic}
Let $k_n$ be the number of loopless matroids on a fixed $n$-element set
whose Kazhdan--Lusztig polynomials are unimodal. Then
\[
\lim_{n\to\infty}\frac{\log k_n}{\log N_n}=1.
\]
\end{coro}
\begin{proof}
A sparse paving matroid with a loop has rank at most one, so there are at
most $2^n$ such matroids. Theorem~\ref{thm:strict-unimodality} therefore gives
$k_n\geq s_n-2^n$.

For completeness, $2^n/s_n\to0$ follows from a simple construction.
For $n\geq4$, put $r=\lfloor n/2\rfloor$ and label the ground set by
$0,\ldots,n-1$. Partition its $r$-subsets according to the sum of their
elements modulo $n$. Two sets in the same part cannot intersect in
$r-1$ elements, since the two remaining elements would then be congruent
modulo $n$. The largest part has at least $\binom nr/n$ members, and each
of its subfamilies defines a different sparse paving matroid by
Lemma~\ref{lem:stable-family-construction}. Hence
\[
s_n\geq 2^{\binom nr/n},
\]
which implies $2^n/s_n\to0$. It follows that
$\log(s_n-2^n)/\log s_n\to1$. The conclusion now follows from
$k_n\leq N_n$ and the theorem of Pendavingh and van der Pol~\cite{PV2015}.
\end{proof}

\section{Non-degeneracy}\label{sec:nondegeneracy}

Following Ferroni and Vecchi~\cite{FV2022}, a matroid of positive rank $d$ is
\emph{non-degenerate} if
\[
\deg P_M(t)=\left\lfloor\frac{d-1}{2}\right\rfloor.
\]

Gedeon, Proudfoot and Young~\cite[Conjecture~2.5]{Gedeon2017}
conjectured that every connected regular matroid is non-degenerate.
For sparse paving matroids, we give a complete classification of the
degenerate cases of positive rank and corank.

The coefficient at index $\lfloor(d-1)/2\rfloor$ determines whether
$M$ is non-degenerate. We first prove that it is positive when $m\geq2$
and $d\geq5$, and then consider the remaining ranks and coranks.

\begin{lem}\label{lem:terminal-positive-spkl}
Let $m\geq2$ and $d\geq5$, and put
\[
r=\left\lfloor\frac{d-1}{2}\right\rfloor.
\]
For every $M\in\mathrm{SP}(d,m+d)$, we have $a(m,d,r)>0$.
\end{lem}

\begin{proof}
Put $C=\binom{m+d}{m}$ and
\[
T_{m,d}=\frac{u(m,d,r)}{\omega(d,r)}.
\]
By \eqref{eq:delta-d-u1}, $\omega(d,r)>0$. The coefficient formula gives
\begin{equation}\label{eq:terminal-factor}
a(m,d,r)=\omega(d,r)\bigl(T_{m,d}-\lambda\bigr).
\end{equation}
It suffices to show that $T_{m,d}$ exceeds the bound for $\lambda$ in
\eqref{eq:lambda-range}.

Suppose first that $d=2k+1$, where $k\geq2$. The sum in
Lemma~\ref{lem:u-fixed-d} has one term, so \eqref{eq:delta-d-u1} gives
\begin{equation}\label{eq:T-odd}
T_{m,2k+1}
=C\frac{m(k+2)}{2(k+m)(k+m+1)}.
\end{equation}
For $2\leq m\leq2k+1$, we have $\lambda\leq C/(2k+2)$ and
\[
T_{m,2k+1}-\frac{C}{2k+2}
=\frac{C(m-1)(k^2+k-m)}
{2(k+1)(k+m)(k+m+1)}>0,
\]
since $k^2+k-m\geq k^2-k-1>0$. For $m\geq2k+1$, the bound is
$\lambda\leq C/(m+1)$, and
\[
T_{m,2k+1}-\frac{C}{m+1}
=\frac{Ck(m^2-3m-2k-2)}
{2(m+1)(k+m)(k+m+1)}>0.
\]
The quadratic factor is increasing in $m$ on this range and equals
$4(k^2-k-1)>0$ at $m=2k+1$.

Now let $d=2k$, where $k\geq3$. The sum in Lemma~\ref{lem:u-fixed-d}
has two terms. The same calculation gives
\begin{equation}\label{eq:T-even}
T_{m,2k}
=C\frac{km(k+2)(k^2+km-m+1)}
{2(k-1)(k+1)(k+m-1)(k+m)(k+m+1)}.
\end{equation}
The identity
\[
k^2+km-m+1=(k-1)(k+m+1)+2
\]
provides a useful lower bound:
\begin{equation}\label{eq:terminal-even-lower}
\frac{T_{m,2k}}{C}>
F_k(m):=\frac{k(k+2)}{2(k+1)}
\frac{m}{(m+k-1)(m+k)}.
\end{equation}
We compare $F_k(m)$ with $1/(2k+1)$ for $m\leq2k$, and with
$1/(m+1)$ for $m\geq2k$.

For $2\leq m\leq2k$, differentiation gives
\[
\frac{\d}{\d m}\frac{m}{(m+k-1)(m+k)}
=\frac{k(k-1)-m^2}{(m+k-1)^2(m+k)^2}.
\]
Thus $m/((m+k-1)(m+k))$ first increases and then decreases, so its
minimum on the interval is at an endpoint. At those endpoints,
\[
F_k(2)=\frac{k}{(k+1)^2},\qquad
F_k(2k)=\frac{k(k+2)}{3(k+1)(3k-1)}.
\]
Both exceed $1/(2k+1)$. Indeed,
\begin{align*}
(2k+1)F_k(2)-1&=\frac{k^2-k-1}{(k+1)^2}>0,\\
(2k+1)F_k(2k)-1
&=\frac{2k^3-4k^2-4k+3}{3(k+1)(3k-1)}>0.
\end{align*}
For the second inequality, set $k=x+3$. Its numerator becomes
$2x^3+14x^2+26x+9>0$.
It follows that $T_{m,2k}>C/(2k+1)\geq\lambda$ on this interval.

For $m\geq2k$, write
\[
(m+1)F_k(m)=\frac{k(k+2)}{2(k+1)}
\frac{m}{m+k}\frac{m+1}{m+k-1}.
\]
Each of the two fractions is increasing in $m$, since $k\geq3$.
Consequently,
\[
(m+1)F_k(m)\geq(2k+1)F_k(2k)>1,
\]
and hence $T_{m,2k}>C/(m+1)\geq\lambda$ on this range as well.
Together with the odd-rank calculation, \eqref{eq:terminal-factor}
proves the lemma.
\end{proof}

Together with the corank-one case and the coefficient formulas in low
rank, this gives the following classification.

\begin{theo}[Non-degeneracy]\label{thm:nondegeneracy}
Let $M$ be a loopless sparse paving matroid of rank $d$ on $m+d$ elements,
with $m,d\geq1$ and $\lambda$ circuit-hyperplanes. Then $M$ is degenerate
if and only if
\[
 m=1,\quad d\geq3\text{ is odd},\quad\lambda=1,
 \qquad\text{or}\qquad(m,d,\lambda)=(4,3,7).
\]
\end{theo}

\begin{proof}
If $d=1$, looplessness gives $\lambda=0$ and $P_M(t)=1$, so $M$ is
non-degenerate. Assume $d\geq2$ and put $r=\lfloor(d-1)/2\rfloor$.
For $m=1$, the bound \eqref{eq:lambda-range} gives
$\lambda\in\{0,1\}$. If $\lambda=0$, then
$P_M(t)=P_{U_{1,d}}(t)$ and its coefficient at index $r$ is positive.
If $\lambda=1$, equation~\eqref{eq:a-coeff} gives
$P_M(t)=P_{U_{1,d-1}}(t)$. The formula in
Lemma~\ref{lem:m-one-mode} shows that its coefficient at index $r$
vanishes exactly when $d\geq3$ is odd.

Now assume $m\geq2$. For $d\leq2$, we have $r=0$ and $P_M(0)=1$.
When $d=3$, equation~\eqref{eq:a-coeff} gives
\[
a(m,3,1)=\frac{m(m+3)}2-2\lambda.
\]
A zero would require $\lambda=m(m+3)/4$. This is not an integer for
$m=2,3$, and it equals $7$ for $m=4$. For $m\geq5$,
\[
\lambda\leq\frac{(m+3)(m+2)}6
<\frac{m(m+3)}4,
\]
where the difference between the last two expressions is
$(m-4)(m+3)/12$. Thus the only degenerate parameter triple in rank three
is $(m,d,\lambda)=(4,3,7)$.

For $d=4$,
\[
a(m,4,1)=\frac{m(m+4)(m+5)}6-3\lambda.
\]
For $m=2$, the bound $\lambda\leq3$ gives $a(2,4,1)\geq5$.
For $m\geq3$, the weaker bound
$\lambda\leq\binom{m+4}{m}/(m+1)$ gives
\[
a(m,4,1)\geq\frac{(m+4)(m^2+5m-18)}{24}>0.
\]
Finally, Lemma~\ref{lem:terminal-positive-spkl} covers every $d\geq5$.

\end{proof}

\begin{remark}
Both kinds of exceptional parameters are realizable. For the corank-one case, let $H$ have
$d$ elements and adjoin a coloop $e$ to the rank-$(d-1)$ uniform matroid
on $H$. Its bases are exactly the $d$-subsets of $H\cup\{e\}$ other than
$H$. Lemma~\ref{lem:stable-family-construction} shows that it is loopless
and sparse paving, with $H$ as its unique circuit-hyperplane.
For the rank-three case, take the seven nonzero vectors of $\mathbb F_2^3$.
The triples $\{x,y,x+y\}$, with $x\ne y$, form seven distinct sets,
and every pair of points belongs to exactly one of them.
Lemma~\ref{lem:stable-family-construction} therefore gives a rank-three
sparse paving matroid with seven circuit-hyperplanes. This is the Fano
matroid, and $a(4,3,1)=14-2\cdot7=0$.
\end{remark}

\appendix
\section{Six coefficient comparisons}\label{app:positivity-certificates}
This appendix records the polynomial comparisons used for coranks two and
three and for the induction on $m$. All six are obtained from explicit
rational expressions in the paper. After the indicated substitutions, each
polynomial has nonnegative coefficients and a positive constant term.
These facts imply positivity on the required parameter regions.

\subsection{The corank induction}
The numerator in \eqref{eq:m-step-numerator} can be written without expanded
recurrence coefficients. Put $n=m+d$ and
\[
\ell_m=d^2-3di-d+3i^2+im+2i.
\]
Then
\begin{align}
W(m,d,i)&=(n+1)(n-i+1)(i+m)\ell_{m+1}\notag\\
&\quad-(n+2)(i+m+2)\bigl(m\ell_m+i(d-i-1)(n+1)\bigr),
\label{eq:m-kernel}\\
N(m,d,i)&=(i+2)(m+1)(d-i-2)(d-i+1)(n+3)W(m,d,i)\notag\\
&\quad+4(i+m)(i+m+1)(i+m+2)Q(d,i)\notag\\
&\qquad\times\bigl((m+1)(n+3)-(n+1)(n-i+1)\bigr).
\label{eq:m-numerator-explicit}
\end{align}
Substituting the formulas for $H,K_m,Q_3$ in
Appendix~\ref{app:b-rec-data} proves these identities.
Define
\begin{align*}
N_+(a,b,c)&=N(a+3,3b+9+c,b+2),\\
N_-(a,b,c)&=-N(a+3,2b+3c+9,b+c+3).
\end{align*}
Write each polynomial as a polynomial in $a$ with coefficients in
$\mathbb Z[b,c]$. For each power of $a$, Table~\ref{tab:orthant-certificates} gives the
number of nonzero coefficients in $(b,c)$, the least such coefficient,
and the constant term. In particular, every entry certifies a polynomial
that is positive for $b,c\geq0$.

\begin{table}[htbp]
\centering\small
\begin{tabular}{@{}ccrrr@{}}
\toprule
Polynomial & Power of $a$ & Terms & Least coefficient & Constant term\\
\midrule
$N_+$ & $0$ & $51$ & $4$ & $7990080$\\
 & $1$ & $51$ & $1$ & $6868576$\\
 & $2$ & $43$ & $1$ & $1621024$\\
 & $3$ & $34$ & $2$ & $118688$\\
 & $4$ & $25$ & $1$ & $2720$\\
\midrule
$N_-$ & $0$ & $54$ & $8$ & $13440000$\\
 & $1$ & $54$ & $4$ & $6739040$\\
 & $2$ & $44$ & $8$ & $896720$\\
 & $3$ & $35$ & $5$ & $38080$\\
 & $4$ & $27$ & $1$ & $400$\\
\bottomrule
\end{tabular}
\caption{Coefficient comparisons for the two corank inductions.}
\label{tab:orthant-certificates}
\end{table}

Thus the expansions of $N_+$ and $N_-$ contain $204$ and $214$ nonzero
terms, respectively, all with positive coefficients.
The expressions \eqref{eq:m-kernel}--\eqref{eq:m-numerator-explicit}
fix the polynomials, including their integer scaling.

\subsection{Coranks two and three}
For clarity, set
\[
D_2(d,i)=2(i+1)(i+3)(d-i+2)Q(d,i),
\]
\[
D_3(d,i)=4(i+1)(i+3)(i+4)(d-i+2)(d-i+3)Q(d,i).
\]
The polynomials in Lemmas~\ref{lem:b-base-m2} and
\ref{lem:b-base-m3} are
\[
N_2=D_2\bigl(b(2,d,i)-1\bigr),\qquad
N_3=D_3\bigl(b(3,d,i)-1\bigr).
\]
They are obtained from the two- and three-term instances of
\eqref{eq:u-def-sum}, using \eqref{eq:delta-quartic}.
\Needspace{10\baselineskip}
The following table summarizes the coefficient data for the four expanded
polynomials used in the proofs.
\begin{center}\small
\begin{tabular}{@{}lrr@{}}
\toprule
Polynomial & Terms & Least coefficient\\
\midrule
$N_2(3x+y+7,x+1)$ & $34$ & $1$\\
$-N_2(2x+3y+4,x+y+1)$ & $33$ & $1$\\
$N_3(3x+y+6,x+1)$ & $51$ & $3$\\
$-N_3(2x+3y+4,x+y+1)$ & $54$ & $2$\\
\bottomrule
\end{tabular}
\end{center}
All four constant terms are positive. The six expansions can be checked
directly from the displayed formulas using exact arithmetic.

\section{The sharper bound on \texorpdfstring{$d=3i+1$}{d=3i+1}}
\label{app:boundary-comparison}
We prove the two inequalities used in Lemma~\ref{lem:boundary-pairs}.
For the first, a quartic has no interior minimum on the relevant interval,
so only its endpoint values are needed. For the second, translation by
the lower endpoint gives a polynomial with positive coefficients.

\begin{lem}\label{lem:boundary-affine}
Let $i\geq2$, $d=3i+1$, and define
\[
L_d(m)=\frac{m+d+2}{2(d+1)},\qquad
T_d(m)=\frac{m+d+2}{2(m+1)}.
\]
Then
\begin{align*}
Q_3L_d(m)+R_3-L_d(m+1)&>0 &&(2\leq m\leq3i),\\
Q_3T_d(m)+R_3-T_d(m+1)&>0 &&(m\geq3i+1).
\end{align*}
Here every recurrence coefficient is evaluated at $(m,3i+1,i)$.
\end{lem}
\begin{proof}
Put
\[
D(i,m)=2(i+m)(i+m+1)(i+m+2)(2i+m+2)(3i+m+2)(8i^2+3i-2).
\]
This denominator is positive for $i\geq2$ and $m\geq2$.
Substituting the recurrence formulas gives quartic polynomials $P_i(m)$
and $S_i(m)$ satisfying
\begin{align}
Q_3L_d(m)+R_3-L_d(m+1)
&=\frac{(3i+m+4)P_i(m)}{(3i+2)D(i,m)},\label{eq:boundary-left-residual}\\
Q_3T_d(m)+R_3-T_d(m+1)
&=\frac{(3i+m+4)S_i(m)}{(m+2)D(i,m)}.\label{eq:boundary-right-residual}
\end{align}
We record below the coefficients and evaluations needed to determine their signs.

Write $P_i(m)=\sum_{j=0}^4 p_j(i)m^j$.
The coefficients needed for its shape are
\begin{align*}
p_1(i)&=(i+1)\bigl(i^3(108i^3-19)+i^2(326i^3-152)
                         +i(274i^3-58)+4\bigr),\\
p_2(i)&=i^2(72i^4-96i^2-102)+i(108i^4-221i^2-1)+6,\\
p_4(i)&=-2i(8i^2+3i-2).
\end{align*}
Thus $p_1,p_2>0$ and $p_4<0$ for $i\geq2$.
For $p_2$, the two bracketed expressions are at least $666$ and $843$,
respectively, using $i^2\geq4$.
For $m>0$,
\[
\frac{P_i'(m)}{m^2}=\frac{p_1(i)}{m^2}+\frac{2p_2(i)}m+3p_3(i)+4p_4(i)m
\]
is strictly decreasing: its derivative is
$-2p_1(i)/m^3-2p_2(i)/m^2+4p_4(i)<0$.
It tends to $+\infty$ as $m\downarrow0$ and to $-\infty$ as
$m\to\infty$. Thus $P_i$ first increases and then decreases, and its
minimum on $[2,3i]$ is attained at an endpoint. The endpoint values are
\begin{align*}
P_i(2)&=(i+2)\bigl(i^3(276i^3-886)+i^2(908i^3-657)
                    +453i^4+2i+24\bigr)>0,\\
P_i(3i)&=4i(2i+1)(3i+1)\bigl(54i^5-67i^4-50i^3+3i^2-22i-8\bigr)>0.
\end{align*}
For the last inequality, putting $i=x+2$ changes the final factor to
\[
54x^5+473x^4+1574x^3+2415x^2+1566x+216.
\]
Hence $P_i(m)>0$ throughout $[2,3i]$.

For the second, write $y=m-3i-1\geq0$. Then
\begin{equation}\label{eq:boundary-tail-expansion}
S_i(3i+1+y)=\sum_{j=0}^4 s_j(i)y^j,
\end{equation}
where the coefficients can be grouped as follows:
\begin{align*}
s_0(i)&=6(2i+1)\bigl((108i^7-6)+10i^6+i(44i^4-65)
                         +338i^4+i^2(192i-81)\bigr),\\
s_1(i)&=(1296i^7-78)+308i^6+i(338i^4-577)
                         +2924i^4+i^2(1706i-643),\\
s_2(i)&=i^5(468i-114)+(68i^4-58)+744i^3+i(146i-213),\\
s_3(i)&=3\bigl(i^4(24i-20)+(8i^3-6)+i(29i-5)\bigr),\\
s_4(i)&=(2i-1)\bigl(2i^2(i-1)+i+2\bigr).
\end{align*}
Every displayed summand is positive for $i\geq2$.
Thus $S_i(m)>0$ for $m\geq3i+1$.
\end{proof}

\section{Explicit recurrence coefficients}\label{app:b-rec-data}

The coefficients in Proposition~\ref{prop:b-rec} have shorter formulas
when the common factors are kept together. We give those formulas here;
they also determine the polynomials used in the preceding appendices.
Put $n=m+d$ and let $Q(d,i)$ be the quartic in \eqref{eq:quartic-Q}.
Define
\begin{align}
H(m,d,i)
&=\frac{m(i+2)(d-i-2)(d-i+1)(n+2)}
{4(i+m)(i+m+1)Q(d,i)}\notag\\[-2pt]
&\qquad\times\bigl(d^2-3di-d+3i^2+im+2i\bigr),
\label{eq:H-compact}\\
K_d(m,d,i)
&=\frac{im(i+2)(d-i-2)(d-i-1)(n+1)(n+2)}
{4(d-2i)(i+m)Q(d,i)},\label{eq:Kd-compact}\\
K_m(m,d,i)
&=\frac{i(i+2)(d-i-2)(d-i-1)(d-i+1)(n+1)(n+2)}
{4(i+m)(i+m+1)Q(d,i)}.\label{eq:Km-compact}
\end{align}
In the notation of the proof of Proposition~\ref{prop:b-rec}, these are
\[
H=\frac{V}{\lambda_{\max}\delta},\qquad
K_d=\frac{n-i+1}{i+1}\frac{E_d}{\lambda_{\max}\delta},\qquad
K_m=\frac{n-i+1}{i+1}\frac{E_m}{\lambda_{\max}\delta}.
\]
Their rational forms follow by canceling the factorials in the uniform
coefficient recurrences and in \eqref{eq:delta-quartic}.

The homogeneous coefficients are
\begin{align}
Q_1(m,d,i)
&=\frac{(d-2i)(d-i+2)(n+3)Q(d,i)}
{(d-i-2)(n+1)(n-i+1)Q(d+1,i)},\label{eq:Q1-compact}\\
Q_2(m,d,i)
&=-\frac{(i+1)(i+3)(d-i-3)(n-i)Q(d,i)}
{(i+2)(d-2i-2)(d-2i-1)(d-i+1)Q(d,i+1)},
\label{eq:Q2-compact}\\
Q_3(m,d,i)
&=\frac{(m+1)(n+3)}{(n+1)(n-i+1)}.\label{eq:Q3-compact}
\end{align}
The inhomogeneous coefficients are
\begin{align}
R_1(m,d,i)
&=H(m,d+1,i)\notag\\
&\quad-Q_1(m,d,i)\bigl(H(m,d,i)+K_d(m,d,i)\bigr),
\label{eq:R1-compact}\\
R_2(m,d,i)
&=H(m,d,i+1)+\frac{i+1}{n-i}Q_2(m,d,i)H(m,d,i),
\label{eq:R2-compact}\\
R_3(m,d,i)
&=H(m+1,d,i)\notag\\
&\quad-Q_3(m,d,i)\bigl(H(m,d,i)+K_m(m,d,i)\bigr).
\label{eq:R3-compact}
\end{align}
Equations~\eqref{eq:D-shift} and \eqref{eq:D-shift-i} give these formulas
directly after normalization. In particular, the quartic in the denominator
of $Q_1$ is $Q(d+1,i)$, and that in the denominator of $Q_2$ is $Q(d,i+1)$.
The parameter ranges in Proposition~\ref{prop:b-rec} ensure that all denominators
are nonzero.

\section{The diagonal recurrence}\label{app:diagonal-recurrences}

For $i\geq2$, we derive the recurrence for $z(i)=b(i+2,3i+2,i)$ used in
Lemma~\ref{lem:critical-sign}. Starting from $b(i+2,3i+2,i)$, apply the
recurrence in $d$ three times, then the recurrence in $m$, and finally the
recurrence in the coefficient index. The successive arguments are
\[
\begin{aligned}
(i+2,3i+2,i)&\longrightarrow(i+2,3i+3,i)
 \longrightarrow(i+2,3i+4,i)\\
&\longrightarrow(i+2,3i+5,i)
 \longrightarrow(i+3,3i+5,i)\\
&\longrightarrow(i+3,3i+5,i+1).
\end{aligned}
\]
Every step lies in the domain of the corresponding recurrence, including
$i=2$. Their composition therefore has the form
\begin{equation}\label{eq:diagonal-affine}
z(i+1)=\alpha_2(i)z(i)+\beta_2(i).
\end{equation}
Substituting the formulas from Appendix~\ref{app:b-rec-data}, we obtain
\begin{align*}
\alpha_2(i)
&=-\frac{(i+1)(i+3)^2(2i+5)^2(4i+9)}
{3(i+2)(2i+1)(2i+3)(3i+5)(3i+7)(4i+5)},\\
\beta_2(i)
&=-\frac{(i+3)^2(2i+5)(28i^2+82i+55)}
{6(i+2)(2i+1)(3i+5)(3i+7)(4i+5)}.
\end{align*}
To check the composition, start with $A=1$ and $B=0$.
At each successive point, the recurrence $b_{\mathrm{new}}=Q_jb+R_j$
replaces the pair $(A,B)$ by $(Q_jA,Q_jB+R_j)$.
After the five shifts, cancellation gives the displayed
$(\alpha_2,\beta_2)$. Since $\alpha_2<0$, the proof of
\eqref{eq:z-bounds} uses both endpoints of the interval $(-3/8,0)$.

\section*{Declaration of AI use}
The main results of this paper were proved in 2023--2024 without
assistance from generative AI. The proofs at that stage relied on
extensive computations using code written entirely by the authors.
We did not submit the manuscript at that time because the proofs were
cumbersome and much of the computational code was included in the text.

During 2025--2026, we used generative AI tools to assist with language
editing, simplify parts of the proofs, and remove computational code
from the manuscript. The main mathematical ideas and overall proof
strategy remain those of the original work.

For the final version, we used these tools to recheck the code and verify
that the boundary cases were covered. We also used them to generate Lean
code to check selected calculations and implications, taking certain cited results as
assumptions. These checks supplement the mathematical proofs; they do
not constitute a complete formalization of the paper in Lean.
The authors take responsibility for the mathematical arguments, the code,
and the final manuscript.

\end{document}